\documentclass[12pt]{amsart}

\numberwithin{equation}{section}
\usepackage{enumerate, xspace}
\usepackage{amssymb}
\usepackage{dsfont}
\usepackage[colorlinks]{hyperref}

\newtheorem{theorem}{Theorem}
\newtheorem{thm}{Theorem}[section]
\newtheorem{lemma}[thm]{Lemma}
\newtheorem{proposition}[thm]{Proposition}
\newtheorem{rem}[thm]{Remark}

\newcommand\bl{{\mathbf l}}
\newcommand\bv{{\mathbf v}}

\newcommand\bB{{\mathbb B}}

\newcommand\bE{{\mathbb E}}
\newcommand\bN{{\mathbb N}}
\newcommand\bP{{\mathbb P}}

\newcommand\bR{{\mathbb R}}

\newcommand\bZ{{\mathbb Z}}

\newcommand\cA{\mathcal{A}}
\newcommand\cB{{\mathcal B}}
\newcommand\cC{{\mathcal C}}

\newcommand\cF{{\mathcal F}}

\newcommand\cL{{\mathcal L}}

\newcommand\cN{{\mathcal N}}
\newcommand\cO{{\mathcal O}}
\newcommand\cP{{\mathcal P}}
\newcommand\cQ{{\mathcal Q}}

\newcommand\Card{{\rm Card}}
\newcommand\Id{{\mathds{1}}}

\newcommand\Var{{\Sigma}}

\def\EXP{\mathbb{E}}
\def\reals{\mathbb{R}}
\def\integers{\mathbb{Z}}

\def\dist{{\rm dist}}
\def\Const{{\rm Const}}

\title[Random Walk in Markovian environment]{Non-perturbative approach to random walk in markovian environment.}
\author{Dmitry Dolgopyat}
\address{Dmitry Dolgopyat\\
Department of Mathematics\\
University of Maryland\\
4417 Mathematics Bldg,  College Park,  MD 20742, USA}
\email{{\tt dmitry@math.umd.edu}}
\urladdr{http://www.math.umd.edu/\~\ \hskip-4pt dmitry}
\author{Carlangelo Liverani}
\address{Carlangelo Liverani\\
Dipartimento di Matematica\\
II Universit\`{a} di Roma (Tor Vergata)\\
Via della Ricerca Scientifica, 00133 Roma, Italy.}
\email{{\tt liverani@mat.uniroma2.it}}
\urladdr{http://www.mat.uniroma2.it/\~\ \hskip-4pt liverani}
\thanks{The first part of this work was written in Tongariro National Park. D.D. thanks T. Chesnokova and D. Kvasov for inviting him there.  We also acknowledge the support of the Erwin Schr\"odinger Institute where this work was completed (during the program {\em Hyperbolic Dynamical Systems}). We thank Ian Melbourne for suggesting reference \cite{Ea}.
We also thank Lise Poncelet, Jean Bricmont and the anonymous referees for several helpful remarks on the first version of this paper.  DD was partially supported by IPST}
\date{}
\begin{document}
\begin{abstract}
We prove the CLT for a random walk in a dynamical  environment where the states of the 
environment at different sites are independent Markov chains.
\end{abstract}
\maketitle
\section{Introduction}
\label{sec:intro}
The study of random walk in random evolving environment has attracted much attention lately. The basic idea is that, due to the time mixing properties of the environment, the CLT should hold in any dimension. Many results have been obtained in the case of transition probabilities close to constant
(\cite{BS, BMP3, BMP4, St, RAS, BZ, DKL} etc) but  
there are still two open problems. On one hand one would like to consider environments with weaker 
mixing properties (some results in this direction have been obtained in \cite{CZ, DKL, DL}).
 On the other hand one would like to understand the case in which the dependence on the environment is not small. 
 The present paper addresses the second issue presenting a new, non-perturbative, approach to
 random walks in evolving environment. 

To make the presentation as transparent as possible we consider a very simple environment: at each site of $\bZ^d$ we have a finite state Markov chain and the chains at different sites are independent.
However, the present argument applies to the situation where
the evolution of each site is described by a Gibbs measure. In fact our approach relies on the fact that the environment
seen from the particle is Gibbsian. We hope that this fact can be established for a wide class of mixing environments
and so it will be useful for the first problem as well.


More precisely, at each site $u\in\integers^d$ consider a Markov chain $\{x_{u}^n\}_{n\in\bN}$ with 
the state space $\cA$ and transition matrix $p_{ab}>0$ for any $a$ and $b$. 
The chains at different sites are independent. Let 
$p_{ab}(k)$ denote the $k$ step transition probability and $\pi_a$
denote the stationary distribution corresponding to $p_{ab}$. Let $\Lambda$ be a finite subset of $\integers^d.$
For each $a\in \cA$ let $q_{a,v}$ be a probability distribution on $\Lambda$. Consider a random walk
$S_n$ such that $S_0=0$ and $S_{n+1}=S_n+v_n$ with probability $q_{x_{S_n}^n, v_n}$. 
Let $\bP$ denote the measure of the resulting Markov process on $\overline{\Omega}:=(\cA^{\bZ^d})^{\bN}\times (\bZ^d)^\bN$ when the environment is started with the stationary measure\footnote{In fact our main result
holds if the environment is started from any  initial measure and the proof requires very little change. We assume that
the initial measure is stationary since this allows us to simplify the notation a little.} 
and the walk starts from zero. We use $\bE$ to denote the associated expectation.
\begin{theorem}
\label{th-main}
For each $d\in\bN$ and random walk $\bP$ as above,
\begin{enumerate}[(a)]
\item  there exists $\bv\in\bR^d$ and a $d\times d$ matrix $\Var\geq 0$ such that 
\[
\begin{split}
&\lim_{n\to\infty} \frac1n \bE(S_n)=\bv\\
&\frac{1}{\sqrt n}\left[S_n-n\bv\right] \Longrightarrow \cN(0,\Var^2) \text{  under }\bP.
\end{split}
\] 
That is, $S_n$ satisfies an averaged (annealed)  Central Limit Theorem. 

\item $\Var>0$ unless there exists a proper affine subspace $\Pi\subset\reals^d$ such that for all $a\in\cA$ we have $q_{a,z}=0$ for $z\not\in \Pi.$
 
\item If $\Var>0$  then $S_n$ satisfies the quenched Central Limit Theorem, that is
 for almost every realization of $\{x_u^n\}$ the distribution of $\frac{S_n-n\bv}{\sqrt n}$ conditioned on
 $\{x_u^n\}$ converges to $\cN(0, \Var^2).$ 
\end{enumerate}
\end{theorem}

\begin{rem}Note that the conditions $p_{ab}>0$ and the independence of the Markov chains at each site can be easily weakened. In fact, one can consider an irreducible Markov chain for which such a condition is verified only for time $n$ transition matrix $p_{ab}(n)$ and/or a situation in which the chains are independent only if at a distance $L$. Then essentially the same proof goes through.\footnote{In the following just look at the system each time $n$ and keep track of all the sites that are at a distance less then $L$ from the sites that the walk can visit in such a time.}
To consider more general environments, more work is needed.
\end{rem}
\section{Gibbs measures.} 
\label{ScGibbs}
Here we collect the information about one dimensional Gibbs measures used in our proof. 
The reader is referred to \cite{Ru} for more information and physics background. Let $\cB$
be a finite alphabet and $\bB$ be a $\Card(\cB)\times \Card(\cB)$ matrix (usually called {\em adjacency matrix}) whose entries are zeroes
and ones. Let $\Omega^+\subset\cB^{\bN} $ be the space of forward infinite sequences $\omega=\{\omega_j\}_{j=0}^\infty$ such that $\bB_{\omega_j \omega_{j+1}}=1$ for all $j$ (the sequences satisfying the last condition are called {\it admissible}). Likewise
we let $\Omega$ and $\Omega^-$ be the spaces of biinfinite and backward infinite admissible
sequences  respectively (in $\Omega^-$ the indices run from $-\infty$ to $0$).

Given $\theta\in (0,1)$ we can define distances $d_\theta$ on 
$\Omega, \Omega^+$ and $\Omega^-$ by 
$$d_\theta(\omega', \omega'')=\theta^k \text{ where }
k=\max\{j: \omega'_i=\omega''_i \text{ for } |i|\leq j\}. $$
We let $\cC_\theta(\Omega^+)$  (respectively $\cC_\theta(\Omega),$ $\cC_\theta(\Omega^-)$)  
denote the space of $d_\theta$-Lipshitz functions $\Omega^+\to\reals.$\footnote{Note that $\cC_\theta(\Omega^+)$ is a Banach space when equipped with the norm
\[
\|f\|_\theta:=|f|_\infty+\sup_{\omega',\omega''\in\Omega^+}\frac{|f(\omega')-f(\omega'')|}{d_\theta(\omega',\omega'')},
\]
where $|f|_\infty=\sup_{\omega'\in\Omega^+}|f(\omega')|$.
The analogous fact holds for $\cC_\theta(\Omega^-)$ and $\cC_\theta(\Omega)$.
}
We say that a function is {\it H\"older} if it belongs to $\cC_\theta$ for some $\theta.$ Let $\tau$ be the shift
map $(\tau\omega)_i=\omega_{i+1}.$ A $\tau$-invariant measure $\mu^+$ on $\Omega^+$
is called {\it Gibbs measure with H\"older potential} if 
\begin{equation}
\label{EqDLR}
 \phi(\xi)=\ln \mu^+(\{\omega_0=\xi_0\}| \omega_1=\xi_1\dots \omega_n=\xi_n\dots) 
\end{equation}
is H\"older. That is the conditional probability to see a given symbol at the beginning of the sequence
depends weakly on the remote future. The function $\phi$ given by \eqref{EqDLR} is called the {\it potential} of $\mu^+.$
In this paper we shall use the phrase `Gibbs measure' to mean Gibbs measure with H\"older potential. The Gibbs
measures for $\Omega^-$ are defined similarly.

If $\mu$ is a $\tau$-invariant measure on $\Omega$ we let $\mu^+$ and $\mu^-$ be its projections
(marginals) to $\Omega^+$ and $\Omega^-$ respectively. Observe that each element of this triple 
determines the other two uniquely. For example given $\mu^+$ we can recover $\mu$ using that, for each $k,m\in\bZ$, $k+m\geq 0$,
\[
\mu(\{\omega\in\Omega:\omega_{i}=\xi_i, -m\leq i\leq k\})=
\mu^+(\{\omega\in\Omega^+:\omega_{i}=\xi_i, 0\leq i\leq k+m\}). 
\]
We will call $\mu$ {\it the natural extension} of $\mu^+.$
\begin{proposition}[\bf Variational Principle](\cite[Theorem 3.5]{PP})
\label{PrVar}
The following are equivalent
\begin{enumerate}[(a)]
\item $\mu^+$ is a Gibbs measure

\item $\mu^-$ is a Gibbs measure

\item There is a H\"older function $\psi:\Omega\to\reals$ such that
$$ \mu(\psi)+h(\mu)=\sup\left(\nu(\psi)+h(\nu)\right) $$
where $h$ denotes the entropy and the supremum is taken over all $\tau$-invariant measures.
\end{enumerate}
\end{proposition}
We will call the measures on $\Omega$ satisfying the conditions of the above proposition Gibbs measures.

If $\phi\in \cC_\theta(\Omega^-)$ we consider {\it transfer operator} on $\cC_\theta(\Omega^-)$ given by
\begin{equation}
\label{EqTO}
(\cL_\phi g)(\omega)=\sum_{\{\varpi\in\Omega^-:\tau^{-1}\varpi=\omega\}} e^{\phi(\varpi)} g(\varpi).  
\end{equation}

\begin{proposition}(\cite[Theorem 2.2]{PP})
\label{PrRPF}
Assume $\theta\in(0,1)$, $\phi\in\cC_\theta(\Omega^-)$ and
\begin{equation}
\label{EqNorm}
\cL_\phi(1)=1.
\end{equation}
Then,
\begin{enumerate}[(a)]
\item $|\cL_\phi(g)|_\infty\leq |g|_\infty.$

\item ({\bf  Ruelle-Perron-Frobenius Theorem})\footnote{In fact the Ruelle-Perron-Frobenius Theorem has a more general version where the condition \eqref{EqNorm} is not required. However that formulation is more
complicated and since we are going to apply this theorem in case $\phi$ is the log of conditional probability
(see \eqref{EqDLR}) the version given by Proposition \ref{PrRPF} is sufficient for our purposes.}
There exist constants $C>0, \gamma<1$ such that
$$\cL_\phi=\cQ+\cP $$
where $\cQ\cP=\cP\cQ=0,$ $\|\cQ^n\|_\theta\leq C\gamma^n $ and
$\cP(g)=\mu^-(g) 1$ where $\mu^-$ is the Gibbs measure with potential $\phi.$\footnote{Given any linear operator $A:\cC_\theta(\Omega^-)\to\cC_\theta(\Omega^-)$ we define, as usual,  $\|A\|_\theta:=\sup_{\|f\|_\theta=1}\|Af\|_\theta$.}
In particular, for each $g\in\cC_\theta(\Omega^-)$,
\[
 |\cL_\phi^n  g-\mu^-(g) 1|_\infty\leq  \|\cL_\phi^n  g-\mu^-(g) 1\|_\theta\leq C \gamma^n \|g\|_\theta. 
\]
\end{enumerate}
\end{proposition}

The next result is a combination of \cite[Theorem 4.13 and Proposition 4.12]{PP} and
\cite[section 4.2a]{Ea} (see also \cite[section 3.6]{Aa}, \cite{Z}).

\begin{proposition}[{\bf CLT in the sense of Renyi}] 
\label{PrCLT}
Let $\mu^+$ be a Gibbs measure and $\nu^+$ 
be a measure absolutely continuous with respect to $\mu^+.$ Let $g$ be H\"older and denote
$G_n(\omega)=\sum_{j=0}^{n-1} g(\tau^j \omega)$ where $\omega$ is distributed according to 
$\nu^+.$

(a) $\frac{G_n-n \mu^+(g)}{\sqrt n}\Longrightarrow \cN(0, \sigma^2)$ where
$ \sigma^2=\mu^+(g^2)+2\sum_{j=1}^\infty \mu^+(g \cdot g\circ \tau^j) . $

(b) $\sigma^2=0$ iff $|G_n-n\mu^+(g)1 |_\infty$ is uniformly bounded.
\end{proposition}

\section{Local environment as seen from the particle}
\label{sec:mix}

To prove Theorem \ref{th-main} we consider the history  $\omega$ of the {\bf local} environment as seen from the particle.\footnote{That is, we record only the environment at the visited sites, contrary to the usual strategy of considering all the environment.} More precisely $\omega_n$ is a pair $(x_{S_n}^n, v_n)$. 
Let $\Omega^+$ be the set of all possible histories. Thus $\Omega^+$ is a space of forward infinite sequences. 
Let $\cB=\{(a,v)\in\cA\times \Lambda\;:\; q_{a,v}>0\}$ be the alphabet for $\Omega$
(in our simple case the adjacency matrix is given by $\bB_{(a', v'), (a'', v'')}=1$, i.e. we have the full shift) and consider the spaces $\Omega$ and $\Omega^-$ defined in Section
\ref{ScGibbs}. We shall use the notation $\bP^+$ 
for the measure induced by $\bP$ on $\Omega^+.$\footnote{Indeed, the map $\Psi:\overline\Omega\to\Omega^+$ defined by $\Psi((x^n,S_n)_{n\in\bN})=(x_{S_n}^n,S_{n+1}-S_n)_{n\in\bN}$ is measurable and onto. Thus $\bP^+(A):=\bP(\Psi^{-1}(A))$. With a slight abuse of notation we will use $\bE$ also for the expectation with respect to $\bP^+$.}

Let $\cF_k$ denote the $\sigma$-algebra generated by $\omega_0,\dots \omega_k$.

\begin{lemma}
\label{lm-mix}
There exists a shift invariant measure $\mu$ on $\Omega$ and $\gamma\in(0,1)$ such that 
for any $k\in\bN$ there is a constant $C_k$ such that
for any $n, m\in\bN$, and any $\cF_k$ measurable function $f$ we have
\[
|\EXP(f\circ\tau^{n}\;|\;\cF_m)-\mu(f)|\leq C_k|f|_\infty \gamma^{n-m}.
\]
In addition, $\mu$ is a Gibbs measure.

Finally, $\bP^+$ is absolutely continuous with respect to $\mu^+$.
\end{lemma}

\begin{proof}[Proof of Lemma \ref{lm-mix}]
Note that the case $m\geq n+k$ is trivially true, we then restrict to  $m<n+k$.

For $\xi\in\Omega^-$, with $\xi_i=(z_i, v_i)$, and $n\in\bN$ let
\begin{equation}\label{eq:per}
\begin{split}
 &e^{\phi(\xi)}=
\begin{cases} p_{z_l, z_0 }(-l)\cdot q_{z_0 v_0} & \text{ if } \;l=\bl(\xi)\neq -\infty \cr
                       \pi_{z_0} \cdot q_{z_0 v_0} & \text{ if } \;\bl(\xi)= -\infty \cr
\end{cases} \\
& e^{\phi_n(\xi)}=
\begin{cases} p_{z_l, z_0 }(-l)\cdot q_{z_0 v_0} & \text{ if } \;l=\bl(\xi)\geq -n \cr
                        \pi_{z_0} \cdot q_{z_0 v_0}  & \text{ if } \;\bl(\xi)< -n \cr
\end{cases}
\end{split} 
\end{equation}
where $\bl(\xi)$ is the last time before $0$ such that $S_l:=\sum_{i=l}^{-1}v_i=0$.\footnote{In other words we
extend $S_n$ to negative $n$ using $\xi_i$ with negative indices and look at the last negative time when the walk visits
zero.} We have
\begin{equation}
\label{LossMemory}
|\phi_n-\phi|_\infty=\cO\left(\lambda^n\right)
\end{equation}
where $\lambda$ is the second eigenvalue of $p_{ab}.$ 
In fact $\phi(\xi)=\phi_n(\xi)$ unless $\bl(\xi)<-n$ 
in which case the distribution of $p_{z_l z_0}(-l)$ is $\cO(\lambda^n)$-close to~$\pi_{z_0}$.

Recall \eqref{EqTO} and 
consider the following transfer operators $\cL:=\cL_\phi,$ $\cL_n:=\cL_{\phi_n}.$
Note that $\cL_n1=\cL1=1$ and that \eqref{LossMemory} implies that there exists $C>0$ such that, for all continuous $g:\Omega^-\to\bR$, 
\begin{equation}
\label{EqTOClose}
|\cL g-\cL_n g|_\infty\leq C\lambda^n|g|_\infty. 
\end{equation}

Next, chose $\xi^* \in\Omega^-$ and for each $b\in\cB$, $m\in\bN$, $\omega^m=(\omega_0,\dots,\omega_m)\in\cB^m$ let $\xi^* \omega^m b\in\Omega^-$ be the concatenation of $\xi^*,$ $\omega^m$ and $b.$\footnote{Essentially $\xi^*$ corresponds to the choice of a `standard past' for any finite sequence.}
Then
\[
 \bP(\{\omega_{m+1}=b\}|\cF_{m})(\omega_0,\dots,\omega_m)=\exp\left(\phi_{m+1}\left(\xi^*\omega^m b \right)\right).
\]
By hypotheses there exists $\tilde f:\cB^{k+1}\to\bR$ such that $f(\omega)=\tilde f(\omega_0,\dots,\omega_k)$.
Observe that if $g:\Omega^+\to\reals$ is a $\cF_{m+1}$ measurable function, then
\begin{equation}\label{eq:compute}
\EXP(g |\cF_m)(\omega_0,\dots,\omega_m)=\sum_{b\in\cB} \bP(\{\omega_{m+1}=b\}|\cF_m)g(\omega_0,\dots,\omega_m,b).
\end{equation}
Therefore for each $l\in\{0,\dots,n+k-m\}$ we have
\begin{equation}\label{eq:condi}
\EXP(f\circ \tau^{n}\;|\;\cF_m)(\omega^m)=
(\cL_{m+1}\cdots \cL_{n+k}\hat f_k)(\xi^*\omega^m),
\end{equation}
where, for $\xi\in \Omega^-$, we define $\hat f_k(\xi)=\tilde f(\xi_{-k},\dots,\xi_0).$ 
Hence, by \eqref{EqTOClose},
\begin{equation}
\EXP(f\circ \tau^{n}\;|\;\cF_m)(\omega^m)
=\cL_{m}\cdots \cL_{n+k-l}\cL^{l}\hat f_k+\cO(\lambda^{n+k-l}|f|_\infty).
\end{equation}
Next, the Ruelle-Perron-Frobenius Theorem (Proposition \ref{PrRPF}(b)) for $\cL$ and the fact that $\cL 1=1$ imply that there are a Gibbs measure $\mu^-$ on $\Omega^-$, and numbers $\theta,\tilde\gamma\in(0, 1)$, $\tilde\gamma\geq \theta$, such that
\[
\EXP(f\circ \tau^{n}\;|\;\cF_m)
=\cL_{m+1}\cdots \cL_{n+k-l}1\cdot \mu^-(\hat f_k)+\cO(\lambda^{n+k-l}+\theta^{-k}\tilde\gamma^{l})|f|_\infty.
\]
Choosing $l=\frac{n-m}2$ (which is smaller than $n+k-m$ by hypotheses) and $\gamma^2=\max\{\lambda,\tilde\gamma\}$ we obtain
$$ \EXP(f\circ \tau^{n}\;|\;\cF_m)
=\mu^-(\hat f_k)+\cO(C_k\gamma^{n-m})|f|_\infty. 
$$
Since $\mu^{-}$ is Gibbs, Proposition \ref{PrVar} implies that $\mu$ and 
$\mu^+$ are Gibbs.

To prove the absolute continuity note that equations \eqref{eq:condi}, \eqref{EqTOClose}
imply for $f\geq 0$ and $\cF_k$ measurable\footnote{The first equality follows by the freedom in the choice of $\xi^*$, 
the last is true because $b$ can take only finitely many values and so $\inf_b\mu^-(\{\omega_0=b\})>0$.}
\[
\begin{split}
\bE(f\;|\;\cF_0)(b)&=\inf_{\xi\in \Omega^-}\cL_{1}\cdots\cL_{k}\hat f_k(\xi b) 
=\inf_{\xi\in \Omega^-}\frac{\cL_{1}\cdots\cL_{k}\hat f_k(\xi b)}{\cL^k\hat f_k(\xi b)}\cL^k\hat f_k(\xi b)\\
&\leq C\inf_{\xi\in \Omega^-} \cL^k\hat f_k(\xi b)\leq C\frac{\mu^-(\Id_{\{\omega_0=b\}}\cL^k\hat f_k)}{\mu^-(\{\omega_0=b\})}\leq C'\mu^+(f). 
\end{split}
\]
Thus $\bE(f)\leq C'\mu^+(f)$, which implies the absolute continuity.
\end{proof}

\section{Proof of Theorem \ref{th-main}} 
\begin{proof}
Part (a) follows from Proposition \ref{PrCLT}(a) and Lemma \ref{lm-mix}.

Next we analyze the possibility that there exists $w\in\bR^d$ such that $\Var^2w=0$. 
In this case Proposition \ref{PrCLT}(b) implies that 
$\langle S_n , w\rangle$ is uniformly bounded. Now, suppose there exist pairs $(a_1, v_1)$
and $(a_2, v_2)$ such that $q_{a_j, v_j}>0$ and $\langle v_1, w\rangle\neq \langle v_2, w\rangle$. 
Let $\xi^{(j)}=\{\omega^j_k\}_{k=0}^\infty,$ $j=1,2$ where $\omega^j_k\equiv (a_j, v_j). $
Then 
$\langle S_n, w\rangle $ cannot be bounded along both the orbits defined by the sequences 
$\xi^{(1)}$ and $\xi^{(2)}.$ This proves (b). 

The quenched CLT was derived from the annealed CLT in \cite[sections 3.2 and 3.3]{DKL}
under the assumption that $q_{a,v}$ was weakly dependent on $a.$ However this assumption 
was not used in this part of \cite{DKL}. Indeed what we need to prove in order to conclude the proof of the present theorem is the equivalent of Theorem 1 in \cite{DKL}. The proof of such a theorem relies only on the mixing of the environment as seen from the particle
 (our Lemma \ref{lm-mix})\footnote{In fact, the process determined by the point of view of the particle is different here from the one used in  \cite{DKL}, as already remarked. Yet, in both cases the walk is an additive functional of such a process and the same arguments apply.}  and the estimate \cite[(2.21)]{DKL}.
 In turn \cite[(2.21)]{DKL} follows from \cite[Lemma 3.2]{DKL} by a general argument that uses only the mixing property of the process (in our case implied by Lemma \ref{lm-mix}) and Assumption (A4) of \cite{DKL} that, in our case, can be replaced by the stronger property

{\em Let $S_n'$ and $S_n''$ be two independent walkers moving in the same environment.
Then conditioned on the event
$$\{\dist\left(S_N', S_N''\right)>m\} $$
the increments of $S_k'$ and $S_k''$ are independent for $k\in \left(N, N+\frac{m}{\Const}\right).$}
(This holds in our case since the chains at different sites are independent.)

Finally, \cite[Lemma 3.2]{DKL} follows from \cite[Lemma 3.3]{DKL}. The only property used in such a derivation is that for any $m$ there exists $N$ such that $S'$ and $S''$ can get $m$ units apart during the time $N$ with positive probability (in our case this follows under the hypothesis that ensure Theorem \ref{th-main}(b)). Such a fact also suffices for the proof of the analogue of \cite[Lemma 3.4]{DKL} and for the derivation of \cite[Lemma 3.3]{DKL} from \cite[Lemma 3.4]{DKL} in the present context.

In conclusion, the proof carries to the present setting without any substantial changes.
\end{proof} 

\end{document}